\documentclass[a4paper,12pt,reqno]{amsart}
\usepackage{amsmath,amssymb,amsthm}
\usepackage{enumerate}
\numberwithin{equation}{section}
\theoremstyle{plain}
\newtheorem{thm}{Theorem}[section]

\theoremstyle{definition}

\newtheorem{exmp}{Example}[section]
\begin{document}
\title[Nilpotency of Elementary Operators on $B(E)$] {Nilpotency of Elementary Operators on $B(E)$}
\author[Gyan Prakash Tripathi and Nand Lal ]{Gyan Prakash Tripathi$^{*,1}$ and Nand Lal\\
         Department of Mathematics\\
       SGR PG College, Dobhi\\
        Jaunpur -- 222149,India\\
        1/89 Vinay Khand,Gomati Nagar \\
        Lucknow-222460,India}
\date{}
\thanks{$^*$E-mail: gptbhu@yahoo.com}

\begin{abstract}
 In this paper, we shall give a necessary and sufficient condition for nilpotency of elementary multiplication operators and some sufficient conditions for elementary operators to be nilpotent on $B(E)$, where E is Banach space.
\end{abstract}
\maketitle

\section{Introduction}
\label{sec:intr}
Let $B(E)$ be the algebra of all bounded linear operators on a Banach 
space $E$, and let $A= (A_1, A_2, \ldots , A_n)$ and $B = (B_1, B_2, \ldots , B_n)$
be two $n$-tuples 
of operators in $B(E)$. The elementary operator $R_{A,B}$ associated 
with $A$ and $B$ is the operator on $B(E)$ into itself defined by
$$R_{A,B} (X) = A_1XB_1 + A_2 XB_2 + \ldots + A_n XB_n$$
for all   $X \in B(E)$ (see \cite{FIAL}).

For $A$ and $B$ in $B(E)$, by $M_{A,B}$ we denote elementary multiplication operator defined by $M_{A,B} (X) = AXB$ for all $X \in B(E)$. This can also be seen as elementary operator of length one. For  $A, B \in B(E),$ inner derivation $\delta_A$ on $B(E)$ into itself is defined by  $\delta_A (X) = AX - XA$ and generalized derivation $\delta_{A,B}$ on $B(E)$ into itself is defined by $\delta_{A,B} (X) = AX - XB$ for all $X \in B(E)$. It is easy to see that generalized derivation and inner derivation are particular cases of elementary operators. 
	
In 1979, Fong and Sourour \cite{FS} investigated the compactness of elementary operators. By using that result they gave the following characterization of nilpotency of generalized derivations. 
\begin{thm}
\cite{FS} Let $X$ be an infinite dimensional Banach space and $S, T \in B(X)$, then the following are equivalent: 
\begin{enumerate}
\item $\delta_{S,T}$ is nilpotent. 
\item There exist a positive integer $n$ such that $\delta_{S,T}^n$ is a compact operator. 
\item There exist a scaler $\lambda$ such that $S-\lambda I$ and $T-\lambda I$ are nilpotent. 
\end{enumerate}
\end{thm}
\section{Results}

In this section,we shall give a necessary and sufficient condition for nilpotency of elementary operators and give some sufficient conditions for elementary operators to be nilpotent. 

\begin{thm}
The elementary multiplication operator $M_{A,B}(X) = AXB$ for all $X \in B(E)$, is nilpotent if and only if either $A$ or $B$ is nilpotent. 
\end{thm}

\begin{proof}
We have $M_{A,B} (X) = AXB$ for all $X \in B(E)$ so $M_{A,B}^n (X) = A^n XB^n$. It is easy to see that $M_{A,B}$ is nilpotent if either $A$ or $B$ is nilpotent. 

Conversely, let $M_{A,B}^n = 0$. Now  $M_{A,B}^n (X) = A^n XB^n = 0$ for all $X \in B(E)$. If $B$ is nilpotent then theorem is done otherwise $B^n z \ne 0$ for some $z \in E$. By consequence of Hahn - Banach theorem there exist a linear functional f such that $f(B^nz)\neq 0.$ Let $f \otimes x$ be rank one operator on $E$ defined by ($f \otimes x) (z) = f(z)x$, where x is a nonzero vector in $E$ (see \cite{MUR}). 

Now 		
\begin{eqnarray*}
		A^n (f \otimes x) B^n = 0&&\Rightarrow A^n (f \otimes x) B^n z = 0\\
		&&\Rightarrow f(B^nz)A^nx = 0\\
				&& \Rightarrow  A^n x = 0.
\end{eqnarray*}
		
If we vary $x$ throughout $E$, then $A^n x = 0$ for all $x \in E$. Hence $A$ is nilpotent. 
\end{proof}

\begin{thm}
Let $R_{A,B}$ be an elementary operator on $B(E)$, where 
$A = (A_1, A_2, \ldots, A_n)$  and $B=(B_1, B_2, \ldots, B_n)$ are $n$-tuples of commuting families in $B(E)$. Then $R_{A,B}$ is nilpotent if either  $A_i$ or $B_i$ is nilpotent for each  $1 \le i \le n$.
\end{thm}

\begin{proof}
We denote $R_{AB} (X) = \sum_{i = 1}^n A_i XB_i,$ for all $X\in B(E).$
	
	For $n=1$, $R_{A,B}$ is an elementary multiplication operator. Therefore statement is true by Theorem $2.1$ 
	
Suppose statement is true for $n = m$. 

Now we show that statement is true for $n = m + 1$ also. 

Let 	

\begin{eqnarray*}
R_{A,B} (X) &=& \sum_{i = 1}^{m + 1} A_i XB_i\\
						& = & \sum_{i = 1}^m A_i XB_i + A_{m + 1} XB_{m + 1}
\end{eqnarray*}
		  
Suppose $R_{A,B}' (X) = \sum_{i = 1}^m A_i XB_i$ and $M_{A_{m+1}, B_{m + 1}} (X) =
A_{m + 1} XB_{m + 1}$. It follows that $R_{A,B} = R_{A,B}' +$ $M_{A_{m + 1}, B_{m + 1}}.$ Since $A_{i}\,^{'}s$ and  $B_{i}\,{'} s$ are commuting, it is easy to see 
that $R_{A,B}' M_{A_{m + 1}, B_{m + 1}} = M_{A_{m + 1}, B_{m + 1}} R_{A,B}'$.  But $R_{A,B}'$ and
$M_{A_{m + 1}, B_{m + 1}}$ are nilpotent. Therefore $R_{A,B} (X) = \sum_{i = 1}^{m + 1}
A_i XB_i$  is nilpotent. 
\end{proof}

\begin{thm}
 Let  $V_{A,B} (X) = AXB - BXA$ be an elementary operator on $B(H)$ into itself, 
 where $AB = BA$. If there exist $\lambda, \mu \in \mathbb{C}$, such that $A - \lambda I$ and 
 $B - \mu I$ are nilpotent then $V_{A,B}$ is nilpotent. 
\end{thm}

\begin{proof}
We have  $V_{A,B} (X) = AXB - BXA$. It is easy to see that
\begin{equation}\label{eq:one}
V_{A - \lambda I, B - \mu I} (X) = V_{A,B} (X) + \lambda \delta_B (X) - \mu \delta_A (X)
\end{equation}
\par
	Since $A-\lambda I$ and $B-\mu I$ are nilpotent, $\delta_A$ and $\delta_B$ are nilpotent follows by Theorem $1.1$. Further $\delta_A \delta_B = \delta_B \delta_A$. Therefore $\lambda
	\delta_B - \mu \delta_A$  is nilpotent. Note that $V_{A - \lambda I, B - \mu I}$ is nilpotent by Theorem 2.2. Since $AB = BA$, it is easy to see that 
	
	$V_{A - \lambda I, B - \mu I}
	(\mu \delta_A - \lambda \delta_B) = (\mu \delta_A - \lambda \delta_B)
	V_{A - \lambda I, B - \mu I}$.
			Therefore from equation (\ref{eq:one}) $V_{A,B} =
	V_{A - \lambda I, B - \mu I} + \mu \delta_A - \lambda \delta_B$ is nilpotent. 
\end{proof}
\section{examples}

\begin{exmp}
Let $\mathcal{A} = M_2 (\mathbb{C})$ be the complex algebra of all $2 \times 2$ matrices  and 
	let \begin{center}$A = \begin{bmatrix}
						0 & &1\\
						0 & &0
						\end{bmatrix}$, $B = \begin{bmatrix}
																			0 & &0 \\
																			1 & &0
																			\end{bmatrix}$.
	\end{center}																		
																			
Note that $AB \ne BA$, $A^2 = B^2 = 0$, $ABA = A$ and $BAB = B$. Consider $V_{A,B} (X) = AXB - BXA$, we get $V_{A,B}^3 = - V_{A,B}$.\\But $V_{A,B} (X) = \begin{bmatrix}
											\mathbf{x}_{22} & 0\\
											0		& \mathbf{-x}_{11}
											\end{bmatrix}$ is a diagonal operator for every $X \in M_2 (\mathbb{C})$, where   
 $X = \begin{bmatrix}
 				\mathbf{x}_{11} & \mathbf{x}_{12}\\
 				\mathbf{x}_{21} & \mathbf{x}_{22}
 			\end{bmatrix}$. Therefore  $V_{A,B}^3 = - V_{A,B} \ne 0$. Hence $V_{A,B}$ is not nilpotent. 
	
	Above example shows that commutativity of $A$ and $B$ cannot be relaxed in Theorem 2.2 and Theorem 2.3. 
	
\end{exmp}

\begin{exmp}
Let $\mathcal{A} = M_3 (\mathbb{C})$ be the complex algebra of all $3 \times 3$ matrices.  
	Suppose  $A = \begin{bmatrix}
									a & b & 1\\
									c & d & 1 \\
									0 & 0 & k
								\end{bmatrix}$
	and  $B = \begin{bmatrix}
						a & b & 0\\
						c & d & 0\\
						0 & 0 & k
						\end{bmatrix}$	    
		such that $a + b = c + d = k \ne 0, b + c\neq 0$.
		
Note that $AB = BA$ and $N = A-B$ is nilpotent. Consider $V_{A,B}(X) = AXB - BXA$, we have $V_{A,B}(X) = NXB - BXN$. Thus $V_{A,B}$ is nilpotent  by Theorem 2.2 but none of $A$ or $B$ is nilpotent. This shows that condition of Theorem 2.1 cannot be extended to elementary operators of length 2. This example also shows that converse of Theorem $2.2$ is not true. Further note that $A - \lambda I$  and $B - \mu I$ are not nilpotent for any $\lambda, \mu \in C$. This shows that converse of Theorem $2.3$ is not true.
\end{exmp}
\textbf {Acknowledgemant}:\\
Research work of first author is supported by CSIR Award No.9/13(951)/2000-EMR-I Dated 19.12.2002. The helpful suggestions of referee are also gratefully acknowledged. 

\end{document}